\documentclass[12pt]{article}

\usepackage{amsmath}
\usepackage{mathtools}
\usepackage{amssymb}
\usepackage{enumerate}
\usepackage{here}
\usepackage{fullpage}
\usepackage[usenames,dvipsnames,svgnames,table]{xcolor}
\usepackage{mathrsfs}
\usepackage{multirow}

\usepackage{cite}
\def\citep{\cite}

\usepackage{dsfont}

\definecolor{lblue}{RGB}{0,110,152}

\definecolor{dred}{RGB}{171,67,53}

\newtheorem{theorem}{Theorem}

\newtheorem{proposition}[theorem]{Proposition}

\newtheorem{define}[theorem]{Definition}

\newtheorem{problem}[theorem]{Problem}

\newcommand{\mendth}{\hfill \ensuremath{\vartriangle}}

\DeclareMathOperator*{\col}{col}

\DeclareMathOperator*{\diag}{diag}

\DeclareMathOperator{\E}{ \mathbb{E}}
\DeclareMathOperator{\p}{ \mathbb{P}}

\newenvironment{proof}{{\it Proof :~}}{\hfill$\diamondsuit$\\}

\begin{document}

\title{A class of $L_1$-to-$L_1$ and $L_\infty$-to-$L_\infty$ interval observers for (delayed) Markov jump linear systems}

\author{Corentin Briat\thanks{corentin@briat.info; http://www.briat.info}}


\date{}

\maketitle

\begin{abstract}
\noindent We exploit recent results on the stability and performance analysis of positive Markov jump linear systems (MJLS) for the design of interval observers for MJLS with and without delays. While the conditions for the $L_1$ performance are necessary and sufficient, those for the $L_\infty$ performance are only sufficient. All the conditions are stated as linear programs that can be solved very efficiently. Two examples are given for illustration.\\

\noindent\textit{Keywords.} Interval observation; Markov jump linear systems; Positive systems; Optimization
\end{abstract}

\section{Introduction}

Interval observers are a particular type of observers that aim at estimating upper and lower bounds on the state value at all times. They have been successfully designed for a wide variety of systems including systems with inputs \citep{Mazenc:11,Briat:15g}, linear systems \citep{Mazenc:12}, delay systems \citep{Efimov:13c,Briat:15g},  LPV systems \citep{Efimov:13b,Chebotarev:15}, discrete-time systems \citep{Mazenc:13,Briat:15g}, impulsive systems \citep{Degue:16nolcos,Briat:17ifacObs,Briat:18_ImtImp} and switched systems \cite{Rabehi:17,Ethabet:17,Briat:18_ImtImp}. To the best of the author's knowledge, no results have been obtained in the context of Markov jump linear systems albeit those systems are important for practical purposes. Those systems are a class of switched systems having the particularity that the switching rule is governed by a continuous-time Markov process with countable finite \cite{Costa:05,Boukas:06} or infinite \cite{Todorov:08} state-space. The positive version of those systems have been studied in considered in \cite{Aitrami:09b,Zhang:14,Bolzern:15} whereas those subject to delays have been considered in \cite{Zhu:17b} where various necessary and sufficient conditions for their stability and performance analysis have been obtained. The importance of positive systems \cite{Farina:00} is that they are instrumental for solving the interval observation problem and that they benefit from very interesting theoretical properties, such as the existence of various necessary and sufficient conditions for their stability and performance characterizations; see e.g. \cite{Farina:00,Briat:11h,Rantzer:15b,Ebihara:11,Colombino:15}.

The goal of this paper is to use state-of-the-art methods for the analysis of positive Markov jump linear systems (MJLS) for the design of interval observers for both MJLS with and without delays. We provide necessary and sufficient conditions for the design of a certain class of interval observers for Markov jump linear systems with delays. Interestingly, the observer can be designed in a way that minimizes the $L_1$-gain on the transfer from the disturbance to the estimation error. The obtained conditions can be checked using linear programming techniques that also allows for the consideration of structural constraints (bounds on the coefficients, zero pattern, etc) on the gains of the observers. Analogous conditions, albeit sufficient only, are provided in the context of the $L_\infty$-gain. Some examples are given for illustration.

\textit{Outline:} The structure of the paper is as follows: in Section \ref{sec:preliminary} preliminary definitions and results are given. Section \ref{sec:perf} is devoted to the performance analysis of positive MJLS. Section \ref{sec:obs} presents the main results of the paper on interval observation. Examples are given in Section \ref{sec:ex}.

\textit{Notations:} The cone of positive and nonnegative vectors of dimension $n$ are denoted by $\mathbb{R}_{>0}^n$ and $\mathbb{R}_{\ge0}^n$, respectively. The notation $\col(x_1,\ldots,x_n)$ denotes the column vector made by stacking the elements $x_1$ to $x_n$ on the top of each other. $\mathds{1}$ denotes the vector of ones.

\section{Preliminaries}\label{sec:preliminary}


Let us consider the following class of positive MJLS:
\begin{equation}\label{eq:mainsyst}
\begin{array}{rcl}
  \dot{x}(t)&=&A_{r_t}x(t)+A_{h,r_t}x(t-h)+E_{r_t}w(t)\\
  z(t)&=&C_{r_t}x(t)+C_{h,r_t}x(t-h)+F_{r_t}w(t)\\
  x(t_0)&=&x_0
\end{array}
\end{equation}
where $x,x_0\in\mathbb{R}_{\ge0}^n$, $u\in\mathbb{R}^{n_u}$, $w\in\mathbb{R}_{\ge0}^{n_w}$ and $z\in\mathbb{R}_{\ge0}^{n_z}$ are the state of the system, the initial condition, the control input, the exogenous input and the performance output, respectively. The disturbance signal $w$ can be either deterministic or stochastic (but independent of $x$ and $r$). This will be further explained when necessary. The stochastic switching signal $r_t\in\{1,\ldots,N\}$ is assumed to be governed by a continuous-time Markov process with discrete state-space. Let $P(\tau)$ defined as $[P(\tau)]_{ij}=p_{ij}(\tau):=\p[r_{s+\tau}=j|r_{s}=i]$. It is known that this matrix solves the forward Kolmogorov equation
\begin{equation}\label{eq:probability}
  \dot{P}(\tau)= P(\tau)\Pi,\ P(0)=I_N
\end{equation}
where the matrix $\Pi$ is Metzler and such that $\Pi\mathds{1}_N=0$.


\begin{proposition}
  The system \eqref{eq:mainsyst} is internally positive, i.e. for all $w(t)\ge0$, then we have that $x(t),z(t)\ge0$, if and only if the matrices $A_i$ are Metzler and the matrices $A_{h,i},E_i,C_i,C_{h,i}$ and $F_i$ are nonnegative for all $i=1,\ldots,N$.
  \end{proposition}

We now define the moment system associated with \eqref{eq:mainsyst} that will play an important role in the rest of the paper:
\begin{define}[Moment system \cite{Zhu:17b}]
  Let $x_i(t):=\E[x(t)\mathds{1}_{r_t=i}]$, $z_i(t):=\E[z(t)\mathds{1}_{r_t=i}]$ and $w_i(t):=\E[w(t)\mathds{1}_{r_t=i}]$. Then, the moment system associated with \eqref{eq:mainsyst} is defined as
  \begin{equation}\label{eq:congruent}
    \begin{array}{lcl}
      \dot{\bar x}(t)&=&\bar A \bar x(t)+\bar A_h \bar x(t-h)+\bar E \bar w(t)\\
      \bar z(t)&=&\bar C \bar x(t)+\bar C_h \bar x(t-h)+\bar F \bar w(t)
    \end{array}
  \end{equation}
  where $\bar x:=\col_i(x_i)\in\mathbb{R}^{Nn}_{\ge0}$, $\bar w:=\col_i(z_i)\in\mathbb{R}^{n_w}_{\ge0}$, $\bar z:=\col_i(z_i)\in\mathbb{R}^{Nn_z}_{\ge0}$ and
  \begin{equation}
    \begin{array}{rclrcl}
      \bar A&:=&\diag_i(A_i)+\Pi^T\otimes I_n,&\bar A_h&:=&\diag_i(A_{h,i})(P(h)^T\otimes I_n)\\
       \bar C&:=&\diag_i\{C_i\},&  \bar C_h&:=&\diag_i\{C_i\}(P(h)^T\otimes I_n)\\
        \bar E&:=&\diag_i\{E_i\},& \bar F&:=&\diag_i\{F_i\}.
    \end{array}
  \end{equation}
  The transfer function of this system is given by
  \begin{equation}
    \bar G(s):=(\bar C+\bar C_h)(sI-\bar A-\bar A_h)^{-1}\bar E+\bar F.
  \end{equation}
\end{define}
This reformulation is different from the one in \cite{Bolzern:15} where conditional moments are considered. The above formulation has the advantage that it does not depend on the value of the probability distribution of the Markov process when deterministic disturbances are considered.
%
%
%

\section{Stochastic stability and performance of (delayed) positive Markov jump linear systems}\label{sec:perf}

\subsection{Stochastic $L_1$ performance of delayed positive Markov jump linear systems}

Let us first define the stochastic $L_1$-gain:
\begin{define}
  The $L_1$-gain of the system \eqref{eq:mainsyst} is defined as the smallest $\xi>0$ such that
  \begin{equation}
    \int_0^\infty\mathds{1}_{q}^T \E\left[z(s)\right]ds\le\xi\int_0^\infty\mathds{1}_{p}^T\E\left[w(s)\right]ds
  \end{equation}
  holds for all $w\in L_1$, $w\ge0$. When the input $w$ is deterministic, the expectation symbol can be removed in the right hand-side.
\end{define}

We then have the following result:
\begin{theorem}[\cite{Zhu:17b}]\label{th:L1}
We assume here that the system  \eqref{eq:mainsyst} is positive and that $w\in L_1$. Then, the following statements are equivalent:
\begin{enumerate}[(a)]
    \item  The system \eqref{eq:mainsyst} with is stochastically stable in the $L_1$-sense and the $L_1$-gain of the transfer $w\mapsto z$ is equal to $\gamma^*>0$.
    \item  The $L_1$-gain $\gamma^*>0$ of the system \eqref{eq:mainsyst} is the optimal value of the linear program
    \begin{equation}
        \gamma^*=\inf_{\gamma>0,\lambda_1>0,\ldots,\lambda_N>0}\gamma
        \end{equation}
        such that $\lambda_i\in\mathbb{R}_{>0}^n$, $i=1,\ldots,N$, and
        \begin{equation}\label{eq:L1cond1}
          \begin{array}{rcl}
            \hspace{-10mm}A_i^T\lambda_i+\sum_{j=1}^N\left(\pi_{ij}I_n+p_{ij}(h)A_{h,j}^T\right)\lambda_j\\
            \qquad \quad+\sum_{j=1}^N\left(p_{ij}(h)C_{h,j}^T+C_i^T\right) \mathds{1}_{n_z}&<&0\\
            E_i^T\lambda_i-\gamma \mathds{1}_{n_w}+F_i^T\mathds{1}_{n_z}&<&0
          \end{array}
        \end{equation}
        for all $i=1,\ldots,N$.
  \item The $L_1$-gain $\gamma^*>0$ of the system \eqref{eq:mainsyst} verifies the expression $  \gamma^*=||\bar G(0)||_1$.  \hfill\mendth
\end{enumerate}
\end{theorem}

\subsection{Stochastic $L_\infty$ performance of positive Markov jump linear systems}

Let us consider now the computation of the $L_\infty$-gain which is given by:
\begin{define}
  The stochastic $L_\infty$-gain of the system \eqref{eq:mainsyst} is defined as the smallest $\xi>0$ such that
  \begin{equation}
    \sup_{t\ge0}\{||\E[z(t)]||_\infty\}]\le\xi\sup_{t\ge0}\{||\E[w(t)]||_\infty\}
  \end{equation}
    holds for all $w\in L_\infty$, $w\ge0$.
\end{define}

We then have the following result:
\begin{theorem}[\cite{Bolzern:15}]\label{th:Linf}
Assume that the non-delayed version of the system \eqref{eq:mainsyst} (i.e. $A_{h,i}=0$ for all $i=1,\ldots,N$)) is positive and that $w\in L_\infty$ is a stochastic signal that is independent of $(x,r)$. Assume further that one of the following equivalent statements hold:
\begin{enumerate}[(a)]
\item The $L_\infty$-gain of the moment system \eqref{eq:congruent} is equal to $\gamma^*$.
\item The $L_\infty$-gain $\gamma^*>0$ of the moment system \eqref{eq:congruent}  is the optimal value of the linear program
    \begin{equation}
        \gamma^*=\inf_{\gamma>0,\lambda_1>0,\ldots,\lambda_N>0}\gamma
        \end{equation}
        such that $\lambda_i\in\mathbb{R}_{>0}^n$, $i=1,\ldots,N$, and
    \begin{equation}
      \begin{array}{rcl}
        A_i\lambda_i+\sum_{j=1}^N\pi_{ji}\lambda_j+E_i \mathds{1}_{n_w}&<&0\\
        C_i\lambda_i-\gamma \mathds{1}_{n_z}+F_i\mathds{1}_{n_w}&<&0
      \end{array}
    \end{equation}
    hold for all $i=1,\ldots,N$.
    \item The $L_\infty$-gain $\gamma^*>0$ of the moment system \eqref{eq:congruent} verifies the expression $\gamma^*=||\bar G(0)||_\infty$.
\end{enumerate}
Then, the $L_\infty$-gain of the system \eqref{eq:mainsyst} with $A_{h,i}=0$, $i=1,\ldots,N$, is at most $\gamma$.
\end{theorem}
This result is not tight in the sense that we only compute an upper-bound on the $L_\infty$-gain of the system \eqref{eq:mainsyst}.

\section{Design of interval observers}\label{sec:obs}

Let us consider now the following system
\begin{equation}\label{eq:systy}
  \begin{array}{rcl}
    \dot{x}(t)&=&A_{r_t}x(t)+A_{h,r_t}x^+(t-h)+E_{r_t}w(t),\ x(0)=x_0\\
    y(t)&=&C_{r_t} x(t)+C_{h,r_t}x^+(t-h)+F_{r_t} w(t)
  \end{array}
\end{equation}
where $x,x_0\in\mathbb{R}^n$, $w\in\mathbb{R}^p$, $y\in\mathbb{R}^r$ are the state of the system, the initial condition, the persistent disturbance input and the measured output. Note that this system is not necessarily positive. We are interested in finding an interval-observer of the form
\begin{equation}\label{eq:obs}
\begin{array}{lcl}
      \dot{x}^\bullet(t)&=&A_{r_t}x^\bullet(t)+A_{h,r_t}x^\bullet(t-h)+E_{r_t}w^\bullet(t)\\
&&\quad+L_{r_t}(y(t)-y^\bullet(t))\\
y^\bullet(t)&=&C_{r_t} x^\bullet(t)-C_{h,r_t}x^\bullet(t-h)-F_{r_t} w^\bullet(t)\\
      x^\bullet(0)&=&x_0^\bullet
\end{array}
\end{equation}
where $\bullet\in\{-,+\}$. Above, the observer with the superscript ``$+$'' is meant to estimate an upper-bound on the state value whereas the observer with the superscript ``-'' is meant to estimate a lower-bound, i.e. $x^-(t)\le x(t)\le x^+(t)$ for all $t\ge0$ provided that $x_0^-\le x_0\le x_0^+$. The signals $w^-,w^+\in L_\infty(\mathbb{R}_{\ge0},\mathbb{R}^p)$ are the lower- and the upper-bound on the disturbance $w(t)$ at any time, i.e. $w^-(t)\le w(t)\le w^+(t)$ for all $t\ge0$. We then accordingly define the following errors  $e^+(t):=x^+(t)-x(t)$ and $e^-(t):=x(t)-x^-(t)$ that are described by the model
\begin{equation}\label{eq:error}
\begin{array}{rcl}
    \dot{e}^\bullet(t)&=&(A_{r_t}-L^\bullet_{r_t}C_{r_t} )e^\bullet(t)\\
    &&+(A_{h,r_t}-L_{r_t}C_{h,r_t} )e^\bullet(t-h)\\
    &&+(E_{r_t}-L_{r_t}F_{r_t} )\delta^\bullet(t)\\
    \zeta^\bullet(t)&=&M_{r_t}e^\bullet
\end{array}
\end{equation}
where $\bullet\in\{-,+\}$, $\delta^+(t):=w^+(t)-w(t)\in\mathbb{R}_{\ge0}^p$ and $\delta^-(t):=w(t)-w^-(t)\in\mathbb{R}_{\ge0}^p$. The matrix $M_{r_t}\in\mathbb{R}^{q\times n}_{\ge0}$ is a nonzero matrix driving the errors $e^\bullet$ to the observed outputs $\zeta^\bullet$. It is assumed to be chosen a priori.

\subsection{A class of $L_1$-to-$L_1$  interval observers}

With all the previous elements in mind, we can state the observation problem that is considered in this section:
\begin{problem}\label{problemL1}
  Find an interval observer of the form \eqref{eq:obs} such that
  \begin{enumerate}[(a)]
    \item The linear systems in \eqref{eq:error} are positive, i.e. $A_i-L_iC_i$ is Metzler and $E_i-L_i F_i $ and $A_{h,i}-L_i C_{h,i}$ are nonnegative for all $i=1,\ldots,N$;
    \item The linear systems in \eqref{eq:error} are stochastically stable in the $L_1$-sense; 
    \item The $L_1$-gain of the transfers $\delta^\bullet\to\zeta^\bullet$, $\bullet\in\{-,+\}$, are minimum.
  \end{enumerate}
\end{problem}

We then have the following result that provides a necessary and sufficient conditions for the existence of a solution to Problem \ref{problemL1}:
\begin{theorem}\label{th:L1obs}
The following statements are equivalent:
\begin{enumerate}[(a)]
  \item There exists an optimal $L_1$-to-$L_1$ interval-observer of the form \eqref{eq:obs} for the system \eqref{eq:systy} that solves Problem \ref{problemL1}.
  %
  \item There exist diagonal matrices $X_i\in\mathbb{R}^{n\times n}$, matrices $U_i\in\mathbb{R}^{n\times r_i}$, $i=1,\ldots,N$, and scalars $\gamma,\alpha$ such that the linear optimization problem
{   \begin{equation}
        \min_{X_1,\ldots,X_N,U_1,\ldots,U_N,\alpha,\gamma} \gamma
      \end{equation}
      subject to the constraints $\alpha,\gamma>0$, $\bar X\mathds{1}>0$,
\begin{equation}\label{eq:L1:cond1}
  \begin{array}{ccc}
\bar X\bar A-\bar U\bar C+\alpha\ge0, \bar X\bar A_h-\bar U\bar C_h\ge0,\bar X\bar E-\bar U\bar F\ge0
  \end{array}
\end{equation}
and
  \begin{equation}\label{eq:L1:cond2}
  \begin{bmatrix}
    \mathds{1}\\
    1
  \end{bmatrix}^T\begin{bmatrix}
   (\bar X\bar A-\bar U\bar C)+(\bar X\bar A_h-\bar U\bar C_h) &&    \bar X\bar E-\bar U\bar F\\
    \mathds{1}^T(I_N\otimes M) && -\gamma \mathds{1}^T
  \end{bmatrix}<0
\end{equation}}
where $\bar X:=\diag_i(X_I)$ and $\bar U:=\diag_i(U_I)$ is feasible.

Moreover, in such a case, if we define $(\bar X,\bar U^*,\alpha^*,\gamma^*)$  as the global minimizer of the above minimization problem, then the optimal gains $L_i^{*}$ are given by $L_i^{*}=(X_i^*)^{-1}U_i^*$.
\end{enumerate}
\end{theorem}
\begin{proof}
  The proof follows from algebraic manipulations using the change of variables $\lambda_i=X_i\mathds{1}$.
\end{proof}

\subsection{A class of $L_\infty$-to-$L_\infty$  interval observers}

The following observation problem will be considered in this section:
\begin{problem}\label{problemLinf}
  Find an interval observer of the form \eqref{eq:obs} such that
  \begin{enumerate}[(a)]
    \item The linear systems in \eqref{eq:error} are positive, i.e. $A_i-L_iC_i$ is Metzler and $E_i-L_iF_i $ is nonnegative for all $i=1,\ldots,N$;
    \item The linear systems in \eqref{eq:error} are stochastically stable in the $L_\infty$-sense; 
    \item The $L_\infty$-gain of the transfers $\delta^\bullet\to\zeta^\bullet$, $\bullet\in\{-,+\}$, is smaller than a certain level $\gamma$ that can be minimized.
  \end{enumerate}
\end{problem}

\begin{theorem}\label{th:Linfobs}
The following statements are equivalent:
\begin{enumerate}[(a)]
  \item There exists a $L_\infty$-to-$L_\infty$ interval-observer of the form \eqref{eq:obs} for the system \eqref{eq:systy} that solves Problem \ref{problemLinf}.
  \item There exist diagonal matrices $X_i\in\mathbb{R}^{n\times n}$, matrices $U_i\in\mathbb{R}^{n\times r_i}$, $i=1,\ldots,N$, and scalars $\gamma,\alpha$ such that the linear optimization problem
{   \begin{equation}
        \min_{X_1,\ldots,X_N,U_1,\ldots,U_N,\alpha,\gamma} \gamma
      \end{equation}
      subject to the constraints $\alpha,\gamma>0$, $\bar X\mathds{1}>0$,
\begin{equation}\label{eq:L1:cond1}
  \begin{array}{ccc}
\bar X\bar A-\bar U\bar C+\alpha\ge0, \bar X\bar E-\bar U\bar F\ge0
  \end{array}
\end{equation}
and
  \begin{equation}\label{eq:L1:cond2}
  \begin{bmatrix}
    \mathds{1}\\
    1
  \end{bmatrix}^T\begin{bmatrix}
   (\bar X\bar A-\bar U\bar C) &&    (\bar X\bar E-\bar U\bar F)\mathds{1}\\
    \mathds{1}^T && -\gamma \mathds{1}^T
  \end{bmatrix}<0
\end{equation}}
where $\bar X:=\diag_i(X_I)$ and $\bar U:=\diag_i(U_I)$ is feasible.

Moreover, in such a case, if we define $(\bar X,\bar U^*,\alpha^*,\gamma^*)$ as the global minimizer of the above minimization problem, then the optimal gains $L_i^*$ are given by $L_i^*=(X_i^*)^{-1}U_i^*$ and is uniformly optimal over all the possible values for $M$ (i.e. it is independent of the values of $M$).
\end{enumerate}
\end{theorem}
\begin{proof}
  The proof follows from the same procedure as in \cite{Briat:15g}. As it is quite long, it is omitted here.
\end{proof}

\section{Examples}\label{sec:ex}

\subsection{A system without delay}

Let us consider the system \eqref{eq:systy} with the matrices
\begin{equation}\label{eq:ex1}
\begin{array}{r}
  A_1=\begin{bmatrix}
    -1 & 0\\1 0.1
  \end{bmatrix}, A_2=\begin{bmatrix}
0.1 & 1\\0 && -2
  \end{bmatrix}, E_1=\begin{bmatrix}
1 & 0.5\\0& 0.5
  \end{bmatrix},\\
   E_2=\begin{bmatrix}
0&0.5\\1&0.5
  \end{bmatrix},
  C_1=C_2=\begin{bmatrix}
   1 & 1
  \end{bmatrix},F_1=F_2=\begin{bmatrix}
   0 & 0
  \end{bmatrix}.
\end{array}
\end{equation}
We also pick $A_{h,1}=A_{h,2}=0$, $C_{h,1}=C_{h,2}=0$ and
\begin{equation}
  \Pi=\begin{bmatrix}
    -2&2\\2&-2
  \end{bmatrix}.
\end{equation}
Solving for the conditions of Theorem \ref{th:L1obs}, we get the gains
\begin{equation}\label{eq:Lex1}
  L_{1} = \begin{bmatrix}
    0\\1
  \end{bmatrix}\textnormal{ and }L_{2} = \begin{bmatrix}
    1\\0
  \end{bmatrix}.
\end{equation}
The inputs are given by $w_1(t)=\sin(t)$, $w_2(t)=\sin(t+\pi/2)$, $w_1^+(t)=w_2^+(t)=1$, $w_1^-(t)=w_2^-(t)=-1$. The $L_1$-gain of the transfer $\delta^\bullet\mapsto \zeta^\bullet$ is equal to 1.0426. Solving now for the conditions in Theorem \ref{th:Linfobs}, we obtain the same gains together with the value 1.1383 as an upper-bound on the $L_\infty$-gain of the transfer $\delta^\bullet\mapsto \zeta^\bullet$. The trajectories of the system and the observers are depicted in Fig.~\ref{fig:L1nondel}.

\begin{figure}[h]
\centering
  \includegraphics[width=0.7\textwidth]{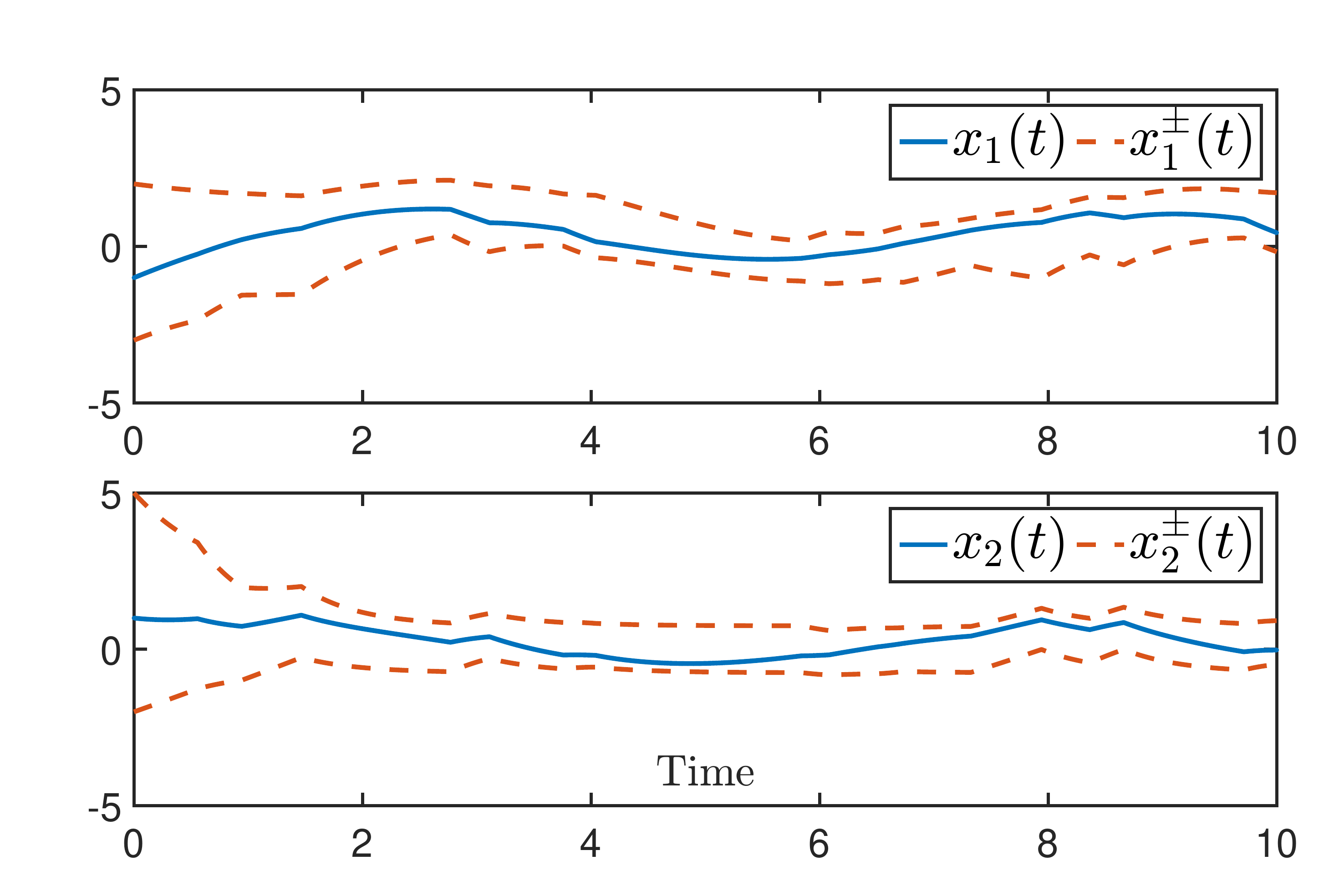}
  \caption{Trajectories of the system \eqref{eq:systy}-\eqref{eq:ex1} and the observer \eqref{eq:obs} with the gains \eqref{eq:Lex1}.}\label{fig:L1nondel}
\end{figure}

\subsection{A system with delay}

Let us consider the system \eqref{eq:systy} with the matrices
\begin{equation}\label{eq:ex2}
\begin{array}{r}
  A_1 = \begin{bmatrix}
    -7.364 &1.065 &1.255\\
1.809 &-9.3& 0\\
0.555 &0 &-7.086
  \end{bmatrix},  A_{h,1} = \begin{bmatrix}
  1.5&3&1.5\\
    2.7 &3 &6.45\\
    0 &1.5 &3
  \end{bmatrix},\\
  A_2 = \begin{bmatrix}
    -7.469 &1.126 &1.3\\
1.851 & -9.222 &0\\
0.6180 &0& -7.171\\
  \end{bmatrix},  A_{h,2} = \begin{bmatrix}
1.8 &3.44 &2.25\\
3 &3.45 &6.75\\
0 &0 &3.6
  \end{bmatrix},\\
   E_1 = \begin{bmatrix}
    1&     0\\
     0&     1\\
     0&     1
  \end{bmatrix},  E_{2} = \begin{bmatrix}
     1&     0\\
     1&     1\\
     1&     0
  \end{bmatrix},C_1=C_2=\begin{bmatrix}
    1 & 0 & 0\\
    0 & 1 & 0
  \end{bmatrix}
\end{array}
\end{equation}
together with $C_{h,1}=C_{h,2}=0$, $F_1=F_2=0$ and
\begin{equation}
  \Pi = \begin{bmatrix}
    -1.5 & 1.5\\
    0.3 & -0.3
  \end{bmatrix}.
\end{equation}
We consider the same inputs as in the other example. We now use Theorem \ref{th:L1obs} to which we add the constraint that the coefficients of the observer gains must not exceed $20$ in absolute value. We get the gains
\begin{equation}\label{eq:Lex2}
  L_{1} = \begin{bmatrix}
    20&    1.0650\\
    1.8090&   20\\
    0.5550 & 0
  \end{bmatrix}\textnormal{ and }L_{2} = \begin{bmatrix}
    1.8&    1.126\\
    1.851&    3.45\\
        0 & 0
  \end{bmatrix}
\end{equation}
which yields the minimum $L_1$-gain 0.88892.  The trajectories of the system and the observers are depicted in Fig.~\ref{fig:L1del}.
\begin{figure}[h]
\centering
  \includegraphics[width=0.7\textwidth]{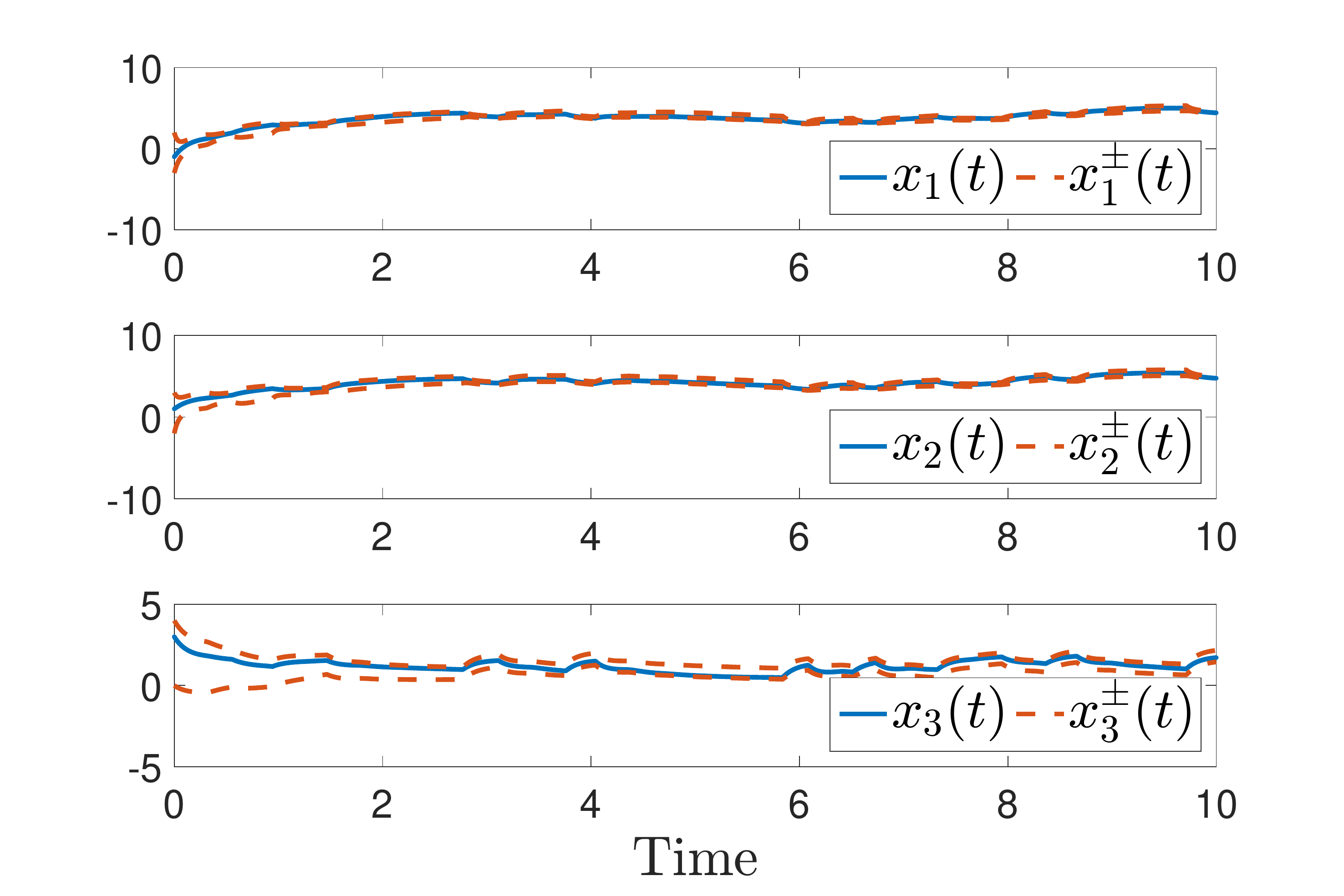}
  \caption{Trajectories of the system \eqref{eq:systy}-\eqref{eq:ex2} and the observer \eqref{eq:obs} with the gains \eqref{eq:Lex2}.}\label{fig:L1del}
\end{figure}

%

\bibliographystyle{plain}

\end{document}